\documentclass[11pt]{amsart}%
\usepackage{amsfonts}
\usepackage{graphicx}
\usepackage{amscd}
\usepackage{amsmath}%
\usepackage{amssymb}

\theoremstyle{plain}
\newtheorem{theorem}{Theorem}[section]

\newtheorem{corollary}[theorem]{Corollary}
\newtheorem{proposition}[theorem]{Proposition}

\theoremstyle{definition}

\newtheorem{remark}[theorem]{Remark}
\newtheorem{example}[theorem]{Example}

\numberwithin{equation}{section}

\begin{document}

 \title[Regularity properties through sequences]{Regularity properties of distributions through sequences of functions}
 \author{Stevan Pilipovi\'c}
 \address{Department of Mathematics and Informatics\\ University of Novi Sad\\ Trg Dositeja Obradovi\'ca 4\\ Novi Sad\\ Serbia.}
 \email{stevan.pilipovic@dmi.uns.ac.rs}
 \author{Dimitris Scarpal\'{e}zos}
 \address{Centre de Math\'{e}matiques de Jussieu\\ Universit\'{e} Paris 7 Denis Diderot\\ Case Postale 7012\\ 2, place Jussieu\\ F-75251 Paris Cedex 05\\ France.}
\email{scarpa@math.jussieu.fr}
 \author{Jasson Vindas}
 \address{Department of Mathematics\\ Ghent University\\ Krijgslaan 281 Gebouw S22\\ B-9000 Gent\\ Belgium.}
\email{jvindas@cage.Ugent.be}
 
\thanks{The research  of S. Pilipovi\'{c} is supported by the Serbian Ministry
of Education and Science, through 
project number 174024}
\thanks{J. Vindas gratefully acknowledges support by a Postdoctoral Fellowship of the Research Foundation--Flanders (FWO, Belgium)}      
  
                 
\begin{abstract}
We give necessary and sufficient criteria for a distribution to be smooth or uniformly H\"{o}lder continuous in terms of approximation sequences by smooth functions; in particular, in terms of those arising as regularizations $(T\ast\phi_{n})$.
\end{abstract}

\maketitle

%%%%%%%%%%%%%%%%%%%%%%%%%%%%%%%%%%%%%%%%%%%%%%%%%%%%%%%%%%%%%%%%%%
\section{Introduction}\label{sec0}
%%%%%%%%%%%%%%%%%%%%%%%%%%%%%%%%%%%%%%%%%%%%%%%%%%%%%%%%%%%%%%%%%%
In this article we provide necessary and sufficient criteria for a distribution to be smooth or have a H\"{o}lderian type regularity. We shall substantially refine and improve earlier results from \cite{her,o,ps}. 

One of the oldest and most useful procedures in analysis is that of \emph{regularization}. It gives a way to study functions and distributions by means of approximations by sequences of regular functions. The procedure has a remarkable importance in the understanding of generalized functions; for instance, it is the essence of the sequential approach to distribution theory \cite{a-m-s,kordist}. It also plays a fundamental role for the theory of generalized function algebras \cite{col84,d-h-p-v,gv,gkos}. The algebras of generalized functions are usually constructed \cite{col84,d-h-p-v,gkos} as quotient algebras of sequences (or nets) of smooth functions. The distributions are then embedded, via regularization, as equivalence classes of sequences. One of the most standard and critical issues in these theories is to find out whether a given generalized function is actually a classical ``smooth'' function in terms of its representative sequences. For example, such a natural question arises when solving (singular) PDE \cite{her,o}. The question has also great interest from the point of view of distribution theory. 

This article is motivated by this general question, and we will provide some answers for H\"{o}lder-Zygmund type and $C^{\infty}$ regularity. Our aim is to describe such regularity properties in terms of growth properties of approximation sequences; in particular, via regularization sequences. We state some samples of our results. Their sharp versions will be the subject of this paper. 

Throughout the article, we use the notation $(\phi_{n})=(\phi_{n})_{n\in\mathbb{N}}$ for a special $\delta$-sequence, also called a sequence of mollifiers, that is,
$\phi_{n}(x)=n^{d}\phi(nx)$, where $\phi\in\mathcal{S}(\mathbb{R}^{d})$ satisfies
$\int_{\mathbb{R}^{d}}
\phi(x)dx=1,$
so that $(\phi_{n})$ is an approximation of the unity. We denote as $C^{\alpha}(\mathbb{R}^{d})$ the global H\"{o}lder space of exponent $\alpha$ \cite{hore,meyer,p-r-v,triebel2006}. The next proposition is a corollary of our Theorem \ref{m2theorem}.

\begin{proposition} 
\label{prop1}
Let $T$ be a distribution with compact support on an open set $\Omega\subset\mathbb{R}^{d}$ and let $(T_{n})$ be a regularization sequence, namely, $T_{n}=T\ast \phi_{n}$. Let $\alpha\in\mathbb{R}_{+}\setminus\mathbb{N}$ and fix an integer $k>\alpha$. A necessary and sufficient condition for $T\in C^{\alpha}(\mathbb{R}^{d})$ is 
\begin{equation}
\label{ireq1}
(\forall m \in {\mathbb
N}^d,|m|\leq k)(\sup_{x\in\Omega}|\partial^{m}T_{n}(x)|=O(n^{k-\alpha})).
\end{equation}
\end{proposition} 

Naturally, we may have used in (\ref{ireq1}) the minimal value $k=[\alpha]+1$. However, the possibility of using different values for $k$ leads to interesting consequences. For instance, using Proposition \ref{prop1}, one easily recovers the ensuing characterization of distributions that are smooth functions, originally due to Oberguggenberger and so useful in the regularity analysis of generalized solutions to partial differential  equations \cite{o}.

\begin{corollary} 
\label{ic1}
Let $T$ be a distribution with compact support on an open set $\Omega$ and let $(T_{n})$ be a regularization sequence. The distribution is smooth, that is, $T\in C^{\infty}(\Omega)$, if and only if there exists $s>0$ such that  
$$
(\forall m \in {\mathbb
N}^d)(\sup_{x\in\Omega}|\partial^{m}T_{n}(x)|=O(n^{s})).
$$
\end{corollary}
\begin{proof} We may assume that $s\notin\mathbb{N}$. Given any $k>0$, write $\alpha=k-s$. Proposition \ref{prop1} yields $T\in C^{k-s}(\mathbb{R}^{d})$. Since this can be done for all $k$, we obtain $f\in C^{\infty}(\mathbb{R}^{d})$. 
\end{proof}

The plan of the article is as follows. We recall in Section \ref{h-zspaces} some well known facts about H\"{o}lder-Zygmund spaces. In Section \ref{sequences} we make comments about convergence rate and growth order of approximation sequences of distributions by functions, both concepts will play an essential role in the rest of the article. We give our main results in Section \ref{regseq}. In Subsection \ref{h-zreg} we present general versions of Proposition \ref{prop1}, which characterize local and global H\"{o}lder-Zygmund regularity of distributions. Subsection \ref{regs} deals with a criterion for smoothness, it extends the one given in Corollary \ref{ic1}. Finally, we discuss other related sufficient conditions for regularity in Subsection \ref{otherreg}.

\subsection{Notation\label{notation}}
We denote by $\Omega$ an open subset of $\mathbb R^d$. We write $\omega\subset\subset\Omega$ if $\omega$ has compact closure contained in $\Omega$. The integral part of $\alpha\in\mathbb{R}$ is denoted as $[\alpha]$. The Schwartz spaces of test functions $\mathcal{D}(\Omega)$, $\mathcal{E}(\Omega)\left(=C^{\infty}(\Omega)\right)$, $\mathcal{S}(\mathbb{R}^{d})$, and their corresponding duals, the spaces of distributions $\mathcal{D}'(\Omega)$, $\mathcal{E}'(\Omega)$, $\mathcal{S}'(\mathbb{R}^{d})$, are well known. The space of $r$-times continuously differentiable functions on $\Omega$ is denoted as $\mathcal{E}^{r}(\Omega)$ (or sometimes simply as $C^{r}(\Omega))$. As in the Introduction, we fix a $\delta$-sequence $(\phi_{n})$, where the test function $\phi\in\mathcal{S}(\mathbb{R}^{d})$ satisfies $\int_{\mathbb{R}^{d}}\phi(x)dx=1$. 

\section{H\"{o}lder-Zygmund spaces}
\label{h-zspaces}
We will measure the regularity of distributions with respect to H\"{o}lder-Zygmund spaces. For the reader's convenience, we collect in this section some background material about these spaces. We start with local H\"{o}lder spaces. Let $\alpha\in\mathbb{R}_{+}\setminus \mathbb{N}$, we say that $f\in C^{\alpha}_{loc}(\Omega)$ if $f\in \mathcal{E}^{[\alpha]}(\Omega)$ and for any $\omega\subset\subset \Omega$,
\begin{equation}
\label{eqH}
\max_{\left|m\right|=[\alpha]}\underset{x\neq t}{\sup_{x,t\in\omega}}\frac{\left|\partial^{m}f(x)-\partial^{m}f(t)\right|}{\left|x-y\right|^{\alpha-[\alpha]}}<\infty.
\end{equation}
The global H\"{o}lder space $C^{\alpha}(\mathbb{R}^{d})$ is defined \cite{hore} by requiring (\ref{eqH}) for $\omega=\mathbb{R}^{d}$ and additionally that $\partial^{m}f\in L^{\infty}(\mathbb{R}^{d})$ for $|m|\leq [\alpha]$.

There are several ways to introduce the global Zygmund space $C^{\alpha}_{\ast}(\mathbb{R}^{d})$ \cite{hore,meyer,triebel2006}. When $\alpha\in\mathbb{R}_{+}\setminus \mathbb{N}$, we have the equality $C^{\alpha}_{\ast}(\mathbb{R}^{d})=C^{\alpha}(\mathbb{R}^{d})$, but the Zygmund spaces are actually defined for all $\alpha\in\mathbb{R}$. They are usually introduced via either a dyadic Littlewood-Paley resolution \cite{triebel2006} or a continuous Littlewood-Paley decomposition of the unity \cite{hore}. However, we shall need a more flexible definition. We follow the approach proposed in \cite[p. 7, Thm. 1.7]{triebel2006} (a continuous version of it is discussed in \cite{p-r-v}).

Let $\alpha\in\mathbb{R}$ and $\varepsilon>0$. We consider two test functions $\theta_1,\theta\in\mathcal{S}(\mathbb{R}^{d})$ with the following compatibility conditions:
\begin{equation}
\label{lpcond1}
|\hat{\theta}_1(u)|>0 \ \mbox{ for } \  \left|u\right|\leq 2\varepsilon,
\end{equation}
\begin{equation}
\label{lpcond2}
|\hat {\theta}(u)|>0  \: \mbox{ for } \:  \varepsilon/2\leq\left|u\right|\leq 2\varepsilon \ \ \mbox{ and } \ \  \int_{\mathbb{R}^{d}}x^{m}\theta(x)dx=0  \: \mbox{ for } \: \left|m\right|\leq[\alpha].
\end{equation} 
When $\alpha<0$, the vanishing requirement over the moments is dropped. We further consider the sequence $(\theta_{2^{j}})_{j\in\mathbb{N}}$ given by
\begin{equation}
\label{lpcond3}
\theta_{2^{0}}=\theta_1 \ \ \ \mbox{while}\ \ \ \theta_{2^{j}}(x)= 2^{jd}\theta(2^{j}x) \ \ \mbox{ for }j\geq 1.
\end{equation}
Then, $C^{\alpha}_{\ast}(\mathbb{R}^{d})$ is the space of all distributions $T\in\mathcal{S}'(\mathbb{R}^{d})$ satisfying:
\begin{equation}
\label{zeq} \left\|T\right\|_{C^{\alpha}_{\ast}(\mathbb{R}^{d})}:=\sup_{x\in\mathbb{R}^{d}, 0\leq j}2^{\alpha j}\left|(T\ast \theta_{2^{j}})(x)\right|<\infty.
\end{equation}
The definition and the norm (\ref{zeq}) are independent of the choice of the sequence as long as (\ref{lpcond1}), (\ref{lpcond2}), and (\ref{lpcond3}) hold. A distribution $T\in\mathcal{D}'(\Omega)$ is then said to belong to $C^{\alpha}_{\ast,loc}(\Omega)$ if for all $\rho\in\mathcal{D}(\Omega)$ we have $\rho T\in C^{\alpha}_{\ast}(\mathbb{R}^{d})$. Clearly, $C^{\alpha}_{\ast,loc}(\Omega)=C^{\alpha}_{loc}(\Omega)$ whenever $\alpha\in\mathbb{R}_{+}\setminus\mathbb{N}$.  

\section{Sequences of smooth functions \label{sequences} }
Our goal in the next section is to describe the regularity of a distribution in terms of approximations to it through sequences of functions. Such regularity properties will depend on two crucial issues: the \emph{rate of convergence} of the approximation and the \emph{growth order} of the sequence with respect to  $n$. We now explain these two concepts.

\subsection{Approximation of distributions via sequences}
Let $T\in \mathcal{D}'(\Omega)$. We shall say that a sequence of locally integrable functions $(f_{n})$ on $\Omega$ is \emph{associated} to the distribution $T$ if $\lim_{n\to\infty}f_{n}=T$ in the weak topology of $\mathcal{D}'(\Omega)$, that is,
\begin{equation}
\label{eqnet1}
(\forall \rho \in \mathcal{D}(\Omega))(\langle T-f_n, \rho
\rangle=o(1),\ n\to\infty).
\end{equation}
In many cases, the rate of approximation in (\ref{eqnet1}) may be much better than just $o(1)$. It is important to keep track of this information. Let $R:\mathbb{N}\to \mathbb{R}_{+}$ be a function such that $R(n)=o(1),$ $n\to\infty.$ We write
$$
T-f_{n}=O(R(n)) \ \mbox{  in } \mathcal{D}'(\Omega)
$$
if 
\begin{equation}
\label{eqnet2}
(\forall \rho \in \mathcal{D}(\Omega))(\langle T-f_n, \rho
\rangle=O(R(n)),\ n\to\infty).
\end{equation}
\begin{example}
\label{ex1} Let $T\in\mathcal{E}'(\Omega)$. Consider the regularization sequence $T_{n}=\left(T\ast \phi_{n}\right)_{|\Omega}$. Then, we have the approximation  $T-T_{n}=O(n^{-b})$ in $\mathcal{D}'(\Omega)$, for any $b\in(0,1]$. It is possible to improve the rate of convergence in this approximation formula by imposing vanishing conditions on the higher order moments of $\phi$. It is not difficult to prove that the assumption $\int_{\mathbb{R}^{d}}t^{m}\phi(x)dx=0$ for each multi-index $1\leq\left|m\right|\leq k$, where $k\in\mathbb{N}$, yields the better approximation rate $T-T_{n}=O(n^{-b})$ in $\mathcal{D}'(\Omega)$ for any $b\in(0,k+1]$.
\end{example}
\begin{example}
\label{ex2} Given a distribution $T$ and a positive function $R$ as above, we can always construct an associated sequence $(f_{n})$ of smooth functions that approximates $T$ as in (\ref{eqnet2}). Let us suppose first that $T\in\mathcal{E}'(\Omega)$. Then $f_{n}=\left(T\ast \phi_{R(n)}\right)_{|\Omega}$,  where $\phi_{R(n)}(x)=(1/R(n))^{d}\phi(x/R(n))$, satisfies (\ref{eqnet2}). The general case $T\in\mathcal{D}'(\Omega)$ follows from a standard partition of the unity argument.
\end{example}

\subsection{Growth of sequences}
We are interested in sequences $(f_n)
$ of $C^{k}$ functions on $\Omega$ for which there exists $s$ such that
\begin{equation}
\label{eqnetg1}
(\forall \omega\subset \subset \Omega)(\forall m \in {\mathbb
N}^d,|m|\leq k)(\sup_{x\in\omega}|\partial^{m}f_n(x)|=O(n^{
s}),
\ n\to\infty).
\end{equation}
If (\ref{eqnetg1}) holds we say that the sequence is of class $(k,s)$; furthermore, we denote as $\mathcal{E}^{k,s}_{\mathbb{N}}(\Omega)$ the set of all sequences of $C^{k}$ functions on $\Omega$ that are of class $(k,s)$. 
The notion makes sense for $k=\infty$, meaning that $(f_{n})$ is a sequence of $C^{\infty}$ functions and that (\ref{eqnetg1}) holds for all $k\in\mathbb{N}$. Observe $\mathcal{E}_{\mathbb{N}}^{k,s}(\Omega)\subseteq \mathcal{E}_{\mathbb{N}}^{k',s'}(\Omega)$ whenever $k'\leq k$ and $s\leq s'$. The intuitive idea behind this notation is to measure the regularity of the sequence in terms of the two parameters: As $k$ increases or $s$ decreases, the sequence becomes more ``regular''. We are particularly interested in the case $s>0$, because, otherwise $(f_{n})$ is associated to the zero distribution.

\section{Main results: regularity through sequences}
\label{regseq}

In other to motivate the results of this section, we start by giving the following standard result.

\begin{proposition} \label{rpropm}Let $T\in\mathcal{D}'(\Omega)$ and let $(f_n)$ be a sequence of $C^{k}$ functions associated to it. Assume that
\begin{equation}
\label{req1}
(\forall \omega\subset \subset \Omega)(\forall m \in {\mathbb
N}^d,|m|\leq k)(\sup_{x\in\omega}|\partial^{m}f_n(x)|=O(1)).
\end{equation}
Then $\partial ^{m}T\in L^{\infty}(\Omega)$ for all $|m|=k $. In particular, if $k\geq d$, then $T$ is a $C^{k-d}$ function on $\Omega$.
\end{proposition}
\begin{proof} The relation (\ref{req1}) gives that, for each $\left|m\right|\leq k$ and $\omega\subset\subset \Omega$, the sequence $((\partial^{m}f_{n})_{|\omega})$ is weakly$^{\ast}$ precompact in $L^{\infty}(\omega)$. The rest is implied by the distributional convergence of $(\partial^{m}f_{n})$ to $\partial^{m}T$.
\end{proof}

Our aim is to weaken the growth constrains in (\ref{req1}), but in such way that one is still able to draw regularity conclusions about the distribution. In order to so, one has to compensate by strengthening the rate of convergence of the sequence. Our three main results go into that direction. Theorem \ref{m2theorem} characterizes H\"{o}lder-Zygmund regularity. Theorem \ref{regas} gives a criterion for smoothness that greatly improves Corollary \ref{ic1} from the Introduction. Finally, Theorem \ref{mtheorem} provides other sufficient conditions for H\"{o}lder-Zygmund regularity.

\subsection{Characterization of H\"{o}lder-Zygmund regularity\label{h-zreg}} We now characterize those compactly supported distributions that belong to a Zygmund space. In this subsection, we restrict our attention to regularization sequences $(T\ast \phi_{n})$. Extensions of the following theorem are indicated in Remarks \ref{remark1} and \ref{remark2} below, where we relax the support assumption and obtain characterizations of $C^{\alpha}_{\ast}(\mathbb{R}^{d})$ and $C^{\alpha}_{\ast,loc}(\Omega)$.

\begin{theorem}
\label{m2theorem} Let $T\in\mathcal{E}'(\Omega)$, $s>0$, and let $(T\ast\phi_{n})$ be a regularization sequence. Then, 
$$((T\ast\phi_{n})_{|\Omega})\in \mathcal{E}^{k,s}_{\mathbb{N}}(\Omega) \Leftrightarrow T\in C_{\ast}^{k-s}(\mathbb{R}^{d}).$$
\end{theorem}
\begin{proof}
Observe \cite{hore} that the partial derivatives continuously act on the Zygmund spaces as $\partial^{m}: C_{\ast}^{\beta}(\mathbb{R}^{d})\mapsto C_{\ast}^{\beta-|m|}(\mathbb{R}^{d})$. Thus, if $T\in C_{\ast}^{k-s}(\mathbb{R}^{n})$ then $\partial^{m}T\in C_{\ast}^{-s}(\mathbb{R}^{d})$ for all $|m|\leq k$. We can then apply \cite[Lem. 5.3, Eq. (5.4)]{p-r-v} and conclude that $(T\ast \phi_{n})\in \mathcal{E}_{\mathbb{N}}^{k,s}(\mathbb{R}^{d})$.

Assume now that $((T\ast\phi_{n})_{|\Omega})\in \mathcal{E}^{k,s}_{\mathbb{N}}(\Omega)$. We first show that actually $(T\ast\phi_{n})\in \mathcal{E}^{k,s}_{\mathbb{N}}(\mathbb{R}^{d})$. Indeed, let $\operatorname*{supp} T\subset\omega_{1}\subset \subset\omega_{2}\subset\subset \Omega$. It suffices to prove that for each multi-index $m\in\mathbb{R}^{d}$
\begin{equation}
\label{eqextra}
\sup_{x\in\mathbb{R}^{d}\setminus \omega_{2} }\left|(\partial^{m}T\ast \phi_{n})(x)\right|=O(1), \ \ \ n>1.
\end{equation}
Let $A$ be the distance between $\overline{\omega}_{1}$ and  $\partial\omega_{2}$. Find $r$ such that  
$$
(\forall \rho\in\mathcal{E}(\mathbb{R}^{d}))(\left|\left\langle \partial^{m}T, \rho\right\rangle\right|<C\|\rho\|_{r,\omega_{1}}),
$$
where $\|\rho\|_{r,\omega_1}=\sup_{u\in \omega_{1}, |p|\leq r}|\partial^{p}\rho(u)|.$ Setting $\rho(u)=n^{d}\phi(n(x-u))$ and using the fact that $\phi$ is rapidly decreasing, we obtain,
$$
\sup_{x\in\mathbb{R}^{d}\setminus \omega_{2} }\left|(\partial^{m}T\ast \phi_{n})(x)\right|< \tilde{C}\sup_{x\in\mathbb{R}^{d}\setminus \omega_{2} }\sup_{u\in \omega_{1}} (1/n + \left|x-u\right|)^{-r-d}\leq \tilde{C}A^{-r-d},
$$
which yields (\ref{eqextra}). Next, set $g_{n}=n^{-s}(T\ast \phi_{n})$. Then, $(T\ast\phi_{n})\in \mathcal{E}^{k,s}_{\mathbb{N}}(\mathbb{R}^{d})$ precisely tells us that $(g_{n})$ is a bounded sequence in $C_{b}^{k}(\mathbb{R}^{d})$, the Banach space of $k$-times continuously differentiable functions that are globally bounded together with all their partial derivatives of order $\leq k$. Since the inclusion mapping $C_{b}^{k}(\mathbb{R}^{d})\mapsto C^{k}_{\ast}(\mathbb{R}^{d})$ is obviously continuous, we obtain that $(g_{n})$ is bounded in the Zygmund space $C^{k}_{\ast}(\mathbb{R}^{d})$. Find $\varepsilon>0$ such that $|\hat{\phi}(u)|>0$ for $|u|\leq 2\varepsilon$. Let $(\theta_{2^{j}})$ be as in (\ref{lpcond1})--(\ref{lpcond3}) (for $\alpha=k$ and this $\varepsilon$). Then, employing the norm (\ref{zeq}), there is $M>0$ such that

$$
\underset{1\leq n,\:0\leq j}{\sup_{x\in\mathbb{R}^{d}}}2^{kj}|(g_{n}\ast\theta_{2^{j}})(x)|=\underset{1\leq n,\:0\leq j}{\sup_{x\in\mathbb{R}^{d}}}2^{kj}n^{-s}|(T\ast\phi_{n}\ast\theta_{2^{j}})(x)|<M.
$$
Setting $n=2^{j}$, $\tilde{\theta}_{1}=\phi\ast\theta_{1}$ and $\tilde{\theta}=\phi\ast\theta$, and noticing that the conditions (\ref{lpcond1})--(\ref{lpcond3}), with $\alpha=k-s$, are fulfilled by $(\tilde{\theta}_{2^{j}})$, we have 
$$
\sup_{x\in\mathbb{R}^{d}, \:0\leq j}2^{(k-s)j}|(T\ast\tilde{\theta}_{2^{j}})(x)|<M,
$$
which in turn implies that $T\in C^{k-s}_{\ast}(\mathbb{R}^{d}) $. 
\end{proof}

We may reformulate Theorem \ref{m2theorem} in order to privilege the role of the Zygmund space. Corollary \ref{m2cor} gives a general form of Proposition \ref{prop1}.

\begin{corollary}
\label{m2cor} Let $T\in\mathcal{E}'(\Omega)$, $\alpha\in\mathbb{R}$, and let $(T\ast\phi_{n})$ be a regularization sequence. If $k\in\mathbb{N}$ is such that $k>\alpha$, then 
$$T\in C_{\ast}^{\alpha}(\mathbb{R}^{d})\Leftrightarrow ((T\ast\phi_{n})_{|\Omega})\in \mathcal{E}^{k,k-\alpha}_{\mathbb{N}}(\Omega) .$$
\end{corollary}
We end this subsection with two remarks.
\begin{remark}
\label{remark1} The proof of Theorem \ref{m2theorem} can be adapted to show the following characterization of the global Zygmund spaces. For a distribution $T\in\mathcal{S}'(\mathbb{R}^{d})$, one has that $f\in C^{\alpha}_{\ast}(\mathbb{R}^{d})$ if and only if, given a $k>\alpha$,
$$
(\forall m \in {\mathbb
N}^d,|m|\leq k)(\sup_{x\in\mathbb{R}^{d}}|\partial^{m}(T\ast\phi_n)(x)|=O(n^{
k-\alpha})).
$$
We leave to the reader the details of such a straightforward modification in the proof of Theorem \ref{m2theorem}. The result just stated improves a theorem of H\"{o}rmann (formulated in \cite{her} by using the language of generalized function algebras).
\end{remark}
\begin{remark}
\label{remark2} One can give a version of Corollary \ref{m2cor} that is valid for all distributions $T\in\mathcal{D}'(\Omega)$. Indeed, by using a partition of the unity, one can construct \cite[Sec. 1.2.2]{gkos} regularization sequences $(T_{n})$ for any distribution $T\in\mathcal{D}'(\Omega)$ such that if $T\in\mathcal{E}'(\Omega)$ one has $(T_{n}-T\ast\phi_{n})\in \mathcal{E}^{\infty,-1}_{\mathbb{N}}(\Omega)$. Thus, given $k>\alpha$, we obtain $T\in C^{\alpha}_{\ast,loc}(\Omega)$ if and only if $(T_{n})\in \mathcal{E}_{\mathbb{N}}^{k,k-\alpha}(\Omega)$.
\end{remark}

\subsection{Characterization of smoothness}\label{regs} We turn our attention to $C^{\infty}$ regularity,
we now provide a criterion of smoothness for distributions. Observe that we already presented a necessary and sufficient condition for smoothness in Corollary \ref{ic1}, that was done in terms of the regularization sequence $(T\ast\phi_{n})$. It turns out that one can employ more general approximation sequences and achieve the same result. The next theorem was originally obtained in \cite{ps}, and extends an earlier result of Oberguggenberger (given within Colombeau theory in \cite{o}). Here we give a new proof based on Theorem \ref{m2theorem}.

\begin{theorem} \label{regas}
 Let $T\in\mathcal D'(\Omega)$ and let $(f_n)$ be a sequence of $C^{\infty}$ functions on $\Omega$ associated to it. Assume that
\begin{equation}\label{reg}(\forall \omega\subset\subset \Omega)(\exists s>0)(\forall m\in \mathbb N^d)
( \sup_{x\in \omega}|\partial^{m}f_n(x)|=O(n^{s})),
\end{equation}
and $(f_{n})$ approximates $T$ with convergence rate:
\begin{equation}
\label{req3}
(\exists b>0)(T-f_{n}=O(n^{-b}) \mbox{ in }\mathcal{D}'(\Omega)).
\end{equation} 
Then $f\in C^\infty(\Omega).$
\end{theorem}
\begin{proof}
Since (\ref{reg}) and the conclusion of Theorem \ref{regas} are local statements, we may assume that $T\in\mathcal{E}'(\Omega)$ and there exists an open subset $\omega \subset \subset \Omega$ such that
\begin{equation} \label{cond}
\mathop{\rm supp} T, \mathop{\rm supp} f_n \subset
\omega,\ n\in\mathbb{N}.
\end{equation}
We will show that $T\in\mathcal{D}(\Omega)$. Our assumption now becomes $(f_n)\in \mathcal{E}^{\infty,s}_{\mathbb{N}}(\Omega)$ for some $s>0$. The support condition (\ref{cond}), the rate of convergence (\ref{req3}), and the equivalence between weak and strong boundedness on $\mathcal{E}'(\Omega)$ (Banach-Steinhaus theorem) yield
\begin{equation}
\label{conclusion1}(\exists r\in \mathbb{N})(\exists C>0)
(\forall \rho\in \mathcal{E}(\Omega))(\forall n\geq 1)(
|\langle T-f_n , \rho \rangle|\leq Cn^{-b}
\|\rho\|_r),
\end{equation}
where $\|\rho\|_r=\sup_{u\in \Omega, |p|\leq r}|\partial^{p}\rho(u)|. $ Let $\alpha$ be an arbitrary positive number. We consider the test function $\phi\in\mathcal{S}(\mathbb{R}^{d})$, recall that we assume $\int_{\mathbb{R}^{d}}\phi(x)dx=1$. Then, by (\ref{reg}) and (\ref{conclusion1}), given any $k\in\mathbb{N}$, we can find positive constants $C_1$ and $C_2$ (depending only on $k$ and $\phi$) such that
$$
\sup_{x\in\omega,\left|m\right|\leq k}\left|\partial^{m}(T\ast\phi_{\nu})(x)\right|\leq C_{1}n^{s}+C_{2}n^{-b}\nu^{d+r+k}, \ \ \  \nu, n\in\mathbb{Z}_{+}.
$$
Find $\eta>0$ such that $\eta s/b<1/2$. Setting $n=[\nu^{k\eta/b}]+1$, we obtain
$$
\sup_{x\in\omega,\left|m\right|\leq k}\left|\partial^{m}(T\ast\phi_{\nu})(x)\right|\leq C_{1}(\nu+1)^{k-k/2}+C_{2}\nu^{k-(\eta k-d-r)}, \ \ \ \nu\in\mathbb{Z}_{+}.
$$
We can now choose $k$ such that $\alpha<\min \left\{k/2,\eta k-d-r\right\}$. The conclusion from the previous estimate is then that $((T\ast \phi_{\nu})_{|\omega})\in \mathcal{E}_{\mathbb{N}}^{k,k-\alpha}(\omega)$, and hence, by Corollary \ref{m2cor}, $T\in C^{\alpha}_{\ast}(\mathbb{R}^{d})$. Since $\alpha$ was arbitrary, it follows that $T\in C^{\infty}(\mathbb{R}^{d})$.
\end{proof}

\subsection{Other sufficient conditions for regularity}\label{otherreg}
The next theorem is directly motivated by Proposition \ref{rpropm}. We relax the growth constrains in (\ref{req1}), and, by requesting an appropriate rate of convergence, we obtain a sufficient condition for the regularity of the distribution.
 \begin{theorem}\label{mtheorem}
 Let $T\in\mathcal{D}'(\Omega)$ and let $(f_n)$ be a sequence of $C^{k}$ functions on $\Omega$ that is associated to it. Assume that either of the following pair of conditions holds:
\begin{enumerate}
\item[\textnormal{(i)}] $(f_{n})\in \mathcal{E}_{\mathbb{N}}^{k,a}(\Omega)$, $\forall a>0$, namely,
\begin{equation}
\label{req2}
(\forall a>0)(\forall \omega\subset \subset \Omega)(\forall m \in {\mathbb
N}^d,|m|\leq k)(\sup_{x\in\omega}|\partial^{m}f_n(x)|=O(n^{a})),
\end{equation}
and the convergence rate of  $(f_{n})$ to $T$ is as in (\ref{req3}).
 
\item [\textnormal{(ii)}] $(f_{n})\in \mathcal{E}_{\mathbb{N}}^{k,s}(\Omega)$ for some $s>0$, and there is a rapidly decreasing function $R:\mathbb{N}\to\mathbb{R}_{+}$, i.e., $(\forall a>0,\:R(n)=O(n^{-a}))$, such that
\begin{equation}
\label{req4}
T-f_{n}=O(R(n))\  \mbox{ in }\mathcal{D}'(\Omega).
\end{equation}
\end{enumerate}
Then, $ T\in C_{*,\:loc}^{k-\eta}(\Omega)$ for every $\eta>0$. 
\end{theorem}

\begin{proof} By localization, it suffices again to assume that $T\in\mathcal{E}'(\Omega)$ and there exists an open subset $\omega \subset \subset \Omega$ such that (\ref{cond}) holds. The proof is analogous to that of Theorem \ref{regas}. As usual, we use the test function $\phi\in\mathcal{S}(\mathbb{R}^{d})$ with $\int_{\mathbb{R}^{d}}\phi(x)dx=1$.

(i) %The assumptions (\ref{req2}) and (\ref{cond}) over $(f_{\varepsilon})_{\varepsilon}$
%give that: 
%\begin{equation} \label{cond11} (\forall a>0)
%(\sup_{\xi\in\mathbb{R}^{d}}(1+|\xi|)^{k}|\hat{f}_\varepsilon(\xi)|=O(\varepsilon^{-a})).\end{equation}
In view of the Banach-Steinhaus theorem, the conditions (\ref{req3}) and (\ref{cond}) imply (\ref{conclusion1}). Thus,
with $C_2=C\sup_{u\in\mathbb{R}^{n},\left|p\right|\leq r}|\partial^{p}\phi(u)|$,
\begin{align*}
\sup_{x\in \omega, |m|\leq k}|\partial^{m}(T\ast \phi_{\nu})(x)|&\leq  C_{2}n^{-b}\nu^{d+r+k}+\|f_{n}\ast\phi_{\nu}\|_{k},
\\
&
\leq
C_{2}n^{-b}\nu^{d+r+k}+\left\|\phi\right\|_{L^{1}(\mathbb{R}^{d})}\|f_{n}\|_{k},
\
\ \ \ n,\nu\in\mathbb{Z}_{+}.
\end{align*}
By (\ref{req2}), given any $a>0$, there exists $M=M_{a}>0$ such that
$$
\sup_{x\in \omega, |m|\leq k}|\partial^{m}(T\ast \phi_{\nu})(x)|\leq  C_{2}n^{-b}\nu^{d+r+k}+Mn^{a},
\
\ \ \ n,\nu\in\mathbb{Z}_{+}.
$$
By taking $n=[\nu^{(k+r+d)/b}]+1$, it follows that
$$\sup_{x\in \omega, |m|\leq k}|\partial^{m}(T\ast \phi_{\nu})(x)|\leq  C_{2}+M (\nu+1)^{a(k+r+d)/b},\ \ \ \nu\in\mathbb{Z}_{+}.
$$
If we take sufficiently small $a$, we conclude that $(T\ast \phi_{n})\in \mathcal{E}_{\mathbb{N}}^{k,\eta}(\omega)$ for all $\eta>0$, and the assertion follows at once from Theorem \ref{m2theorem}.

(ii) The relation (\ref{req4}), the fact that $R$ is rapidly decreasing, and the Banach-Steinhaus theorem imply 
$$
(\exists r \in \mathbb{N})(\forall a>0) (\forall \rho \in
\mathcal{E}^r(\Omega)) (|\langle T-f_n, \rho
\rangle|=O(n^{-a})).
$$
As in part (i),
we have
$$\sup_{x\in \omega, |m|\leq k}|\partial^{m}(T\ast \phi_{\nu})(x)|\leq  Cn^{-a}\nu^{d+r+k}+ \left\|\phi\right\|_{L^{1}(\mathbb{R}^{d})}\|f_{n}\|_{k},\ \ \ n,\nu\in\mathbb{Z}_{+}.
$$
for some constant $C=C_{a}$. Since
$(f_n)\in \mathcal{E}^{k,s}_{\mathbb{N}}(\Omega)$, there is another constant $C=C_{a,s,\phi}>0$ such that
$$
\sup_{x\in \omega, |m|\leq k}|\partial^{m}(T\ast \phi_{\nu})(x)|\leq  Cn^{-a}\nu^{d+r+k}+Cn^{s},\ \ \ n,\nu\in\mathbb{Z}_{+}.
$$
Setting $n=[\varepsilon^{(k+r+d)/a}]+1$, we have
\begin{equation*} 
\sup_{x\in \omega, |m|\leq k}|\partial^{m}(T\ast \phi_{\nu})(x)|\leq  C+C(\nu+1)^{s(k+r+d)/a},\ \ \ \nu\in\mathbb{Z}_{+}.
\end{equation*}
Thus, taking large enough $a>0$, one establishes $T\in\mathcal{E}^{k,\eta}_{\mathbb{N}}(\omega)$ for all $\eta>0$. The conclusion $T\in C_{\ast}^{k-\eta}(\mathbb{R}^{d})$ follows once again from Theorem \ref{m2theorem}.
\end{proof}

We conclude this article with several comments about Theorem \ref{mtheorem}.

The hypotheses (\ref{req3}) and (\ref{req4}) are essential ingredients in Theorem \ref{mtheorem}. The next two examples illustrate the fact that none of them can be omitted.

\begin{example}
\label{rex3} Consider $(f_{n})=(\phi_{\log n})$, i.e., the sequence given by $f_{n}(x)= (\log n)^{d}\phi(x\log n)$. Clearly, $(f_{n})\in \mathcal{E}^{\infty,s}_{\mathbb{N}}(\mathbb{R}^{d})$, $\forall s>0$. Moreover, this net is associated to $\delta$, the Dirac delta distribution. What makes fail the conclusion of Theorem \ref{mtheorem} in this example is the fact that the rate of convergence of $(f_{n})$ is too slow: if the rate of convergence where slightly faster, as in (\ref{req3}), we would be able to conclude that $T$ is a smooth function! 
\end{example}

\begin{example}
\label{rex4}
Let $T\in\mathcal{E}'(\mathbb{R}^{d})$ and $s>0$. Suppose that $T\in C^{k-s}_{\ast}(\mathbb{R}^{d})$ but $T\notin C^{k-s/2}_{\ast}(\mathbb{R}^{d})$. By Theorem \ref{m2theorem}, $(T\ast \phi_{n})\in \mathcal{E}_{\mathbb{N}}^{k,s}(\mathbb{R}^{d})$. However, the conclusion of Theorem \ref{mtheorem} fails for $T$. In this case, the approximation rate is much slower than (\ref{req4}), even if one assumes vanishing of the higher order moments of $\phi$ (cf. Example \ref{ex1}).  
\end{example}

When $T\in\mathcal{E}'(\Omega)$, we may employ in part (i) of Theorem \ref{mtheorem} the regularization sequence $f_{n}=(T\ast \phi_{n})_{|\Omega}$; however, for this case it is better to apply Theorem \ref{m2theorem}, because it provides the optimal regularity conclusion.

%%%%%%%%%%%%%%%%%%
%%%%%%%%%%%%%%%%%
%%%%%%%%%%%%%%%%%%%%%%%%%%%%%%%%%%%%%%%%%%%%%%%%%%%%%%%%%%%%%%%%%%
%%%%%%%%%%%%%%%%%%%%%%%%%%%%%%%%%%%%%%%%%%%%%%%%%%%%%%%%%%%%%%%%%%

%%%%%%%%%%%%%%%%%%%%%%%%%%%%%%%%%%%%%%%%%%%%%%%%%%%%%%%%%%%%%%%%%%
%%%%%%%%%%%%%%%%%%%%%%%%%%%%%%%%%%%%%%%%%%%%%%%%%%%%%%%%%%%%%%%%%%
 \end{document}